\documentclass[oneside,english]{amsart}
\usepackage[T1]{fontenc}
\usepackage[latin9]{inputenc}
\usepackage{units}
\usepackage{mathrsfs}
\usepackage{mathtools}
\usepackage{amsbsy}
\usepackage{amstext}
\usepackage{amsthm}
\usepackage{amssymb}
\usepackage{stmaryrd}
\usepackage{stackrel}
\usepackage{graphicx}
\usepackage{geometry}
\usepackage{array} 
\makeatletter

\newcommand{\noun}[1]{\textsc{#1}}
\providecommand{\tabularnewline}{\\}

\numberwithin{equation}{section}
\numberwithin{figure}{section}
\theoremstyle{plain}
\newtheorem{thm}{\protect\theoremname}
\newtheorem{prop}[thm]{\protect\propositionname}
\theoremstyle{definition}
\newtheorem{defn}[thm]{\protect\definitionname}
\newtheorem{rem}[thm]{\protect\remarkname}

\usepackage{babel}
\usepackage{bbold}
\usepackage[all]{xy}

\makeatother

\usepackage{babel}
\providecommand{\definitionname}{Definition}
\providecommand{\propositionname}{Proposition}
\providecommand{\remarkname}{Remark}
\providecommand{\theoremname}{Theorem}

\begin{document}
\title{\noun{Integral elements of Okubo algebra and the E8-lattice}}
\author{Daniele Corradetti}
\address{Grupo de F\'isica Matem\'atica\\
Instituto Superior T\'ecnico\\
Av. Rovisco Pais, 1049-001 Lisboa, Portugal}
\email{danielecorradetti@tecnico.ulisboa.pt}
\begin{abstract}
In this work we study the interplay between the Coxeter-Dickson $E_{8}$-order,
the para-octonions, and the real Okubo algebra. We first explain why a usual
rotation presentation of the order-three automorphism, although algebraically
correct, does not preserve the classical integral orders. We then introduce an
integral order-three automorphism and prove that its Petersson isotope is the
compact real Okubo algebra. The four octonionic orders with lattices $C_{8}$,
$A_{2}^{4}$, $D_{4}^{2}$, and $E_{8}$ are preserved by the new realisation of the Okubo product. In
particular, the same Coxeter-Dickson $E_{8}$ lattice supports octonionic,
para-octonionic, and Okubo integral structures, although the resulting
multiplicative systems are not isomorphic.
\end{abstract}

\maketitle
\tableofcontents{}

\section*{\noun{introduction and motivations}}

Integral numbers in division algebras are a classical and fascinating topic at
the \emph{interplay} of algebra, arithmetic, and exceptional geometry. Already
in the XIX century, Gauss \cite{Gau} and Hurwitz \cite{Hu19} extended the notion
of integers to the complex numbers and the quaternions, identifying the Gaussian
integers and the Hurwitz quaternions as natural integral structures. The culmination
of this classical line of research is the \emph{Coxeter-Dickson octonions}
$\mathbb{O}_{E_{8}}$ \cite{Co46,Di23}, the maximal order in the algebra of the
octonions $\mathbb{O}$, which is isometric to the exceptional lattice $E_{8}$.
This octonionic background is also close to modern exceptional geometry and to
the systematic study of octonionic Rosenfeld spaces
\cite{MCCAI23-Rosenfeld}.

The landscape of eight-dimensional real division algebras, however, extends well
beyond the octonions. As a consequence of the Generalised Hurwitz Theorem
\cite{ElDuque Comp}, among the eight-dimensional composition algebras over
$\mathbb{R}$ there are three non-isomorphic division cases: the \emph{octonions}
$\mathbb{O}$ (unital), the \emph{para-octonions} $p\mathbb{O}$ (para-unital), and
the \emph{real Okubo algebra} $\mathcal{O}$ (non-unital). While octonions have a
unit element and para-octonions have a para-unit, the Okubo algebra merely contains
\emph{idempotent} elements, and its automorphism group is $\mathrm{SU}(3)$ rather
than $\mathrm{G}_{2}$. A synoptic comparison of these algebras is given in
Table~\ref{tab:Synoptic-table-of}.
This non-unital direction has recently appeared in the study of Okubo projective
geometry, Okubo Spin groups, minimal Cayley-plane realisations, collineation
groups of octonionic planes, and possible physical uses of the Okubo algebra
\cite{CZ22-OkuboSpin,CMZ24-MinimalCayley,CMZ25-Collineations,MCZ25-OkuboPhysics}.

The question of integral elements for $p\mathbb{O}$ and $\mathcal{O}$ has, to
our knowledge, not been systematically separated from the classical unital case.
The definition due to Dickson \cite{Di23} and refined by Johnson
\cite{Jo13,Jo17,Jo18} requires the presence of a unit element, and therefore
does not directly apply to non-unital algebras such as $p\mathbb{O}$ and
$\mathcal{O}$. In this work we isolate the part of Johnson's definition that
survives for isotopes: closure under the relevant product, integrality of norm
and trace, and the relation with a discrete lattice or module.
For complementary viewpoints on eight-dimensional composition algebras, the
Cayley plane, and their recovery from geometric algebra, see also
\cite{Corr26-8DComposition,Corr24-3DGA}.

The para-octonionic case gives a clean result. Since the Coxeter-Dickson order
$\mathbb{O}_{E_{8}}$ is closed under octonionic conjugation and multiplication,
it is also closed under $x\bullet y=\overline{x}\cdot\overline{y}$; hence the
same $E_{8}$ lattice supports a $\mathbb{Z}$-integral para-octonionic structure.
This is a perfect analogy with the octonionic case.

The Okubo case is subtler. A common Petersson presentation uses a rotation by
$2\pi/3$ and consequently displays coefficients in $\mathbb{Q}(\sqrt{3})$.
That automorphism produces the real Okubo algebra, but it does not preserve the
standard octonionic orders. The obstruction is therefore arithmetic rather
than algebraic. We replace this rotation by an integral signed permutation
$\tau$ of the Fano basis. Its fixed subalgebra is quaternionic, so the
classification of order-three Petersson algebras shows that the resulting
isotope is again the compact real Okubo algebra. At the same time, $\tau$
preserves all four crystallographic octonionic orders considered here.

The main consequence is simple: the Coxeter-Dickson order itself is closed
under the new Okubo product. Thus the octonionic and Okubo constructions have
exactly the same additive $E_{8}$ lattice, the same norm, and the same $240$
vectors of norm one. Their products are nevertheless different: the
octonionic roots form a Moufang loop, whereas the Okubo roots form a finite
quasigroup with distinguished idempotents.

The work is organised as follows. In Sec.~1 we review the classical notion of
integral elements over division algebras following Dickson and Johnson. In
Sec.~2 we review the three algebras $\mathbb{O}$, $p\mathbb{O}$, and
$\mathcal{O}$ and their interplays. In Sec.~3 we exhibit the failure of the old
rotation, present the integral automorphism, and prove preservation of the four
integral systems. In Sec.~4 we compare the octonionic and Okubo structures on
the common $E_{8}$ lattice. Finally, in Sec.~5 we discuss our findings and
outline future directions.

\section{\noun{what is an integer }}

Modern research on integral numbers started in the XIX century with
Gauss, Kronecker and Frege, culminating in the works of Dedekind,
Peano \cite{De88,Pe89} on the axiomatization of natural and integer
numbers. On the other hand, attempts to generalise the notion of integers
in the realm of division algebras such as $\mathbb{C}$ complex numbers
and $\mathbb{H}$ quaternions were done by Gauss\cite{Gau}, Hamilton
and Hurwitz \cite{Hu19}. From that chain of research the definition
that settled in the scientific community was that of Leonard Dickson
in 1923 \cite[p. 141-142]{Di23} which we briefly review here. Given
a composition algebra $\mathbb{A}$, i.e. an algebra with a norm $n$
such that $n\left(x\cdot y\right)=n\left(x\right)n\left(y\right)$,
endowed with an involution $x\longrightarrow\overline{x}$ called
\emph{conjugation}, then a set $I$ of \emph{integral elements} of
$\mathbb{A}$ is a set that is closed under addition and multiplication
that contains the element $1$ and such that for every element $x\in I$
we have 
\begin{align}
tr\left(x\right)=x+\overline{x} & \in\mathbb{Z},\label{eq:trace1}\\
n\left(x\right)=x\overline{x} & \in\mathbb{Z}.\label{eq:norm1}
\end{align}
Although this definition is broad and widely accepted, it has two
notable shortcomings. First, it can easily yield an infinite number
of integral elements over Hurwitz division algebras, encompassing
the real numbers $\mathbb{R}$ real numbers, the $\mathbb{C}$ complex
numbers, the $\mathbb{H}$ quaternions and the $\mathbb{O}$ octonions.
For example, for any natural number $m\in\mathbb{N}$, the subring
$\mathbb{Z}\left[\sqrt{-m}\right]$ is a set of integral elements
of $\mathbb{C}$. Second, its specificity -requiring the presence
of the unit element- automatically excludes integral systems from
non-unital algebras, such as para-octonions $p\mathbb{O}$ and Okubo
algebra $\mathcal{O}$.

The first issue with Dickson's definition was addressed by Johnson
\cite{Jo13,Jo17,Jo18}. Inspired by the work of Coxeter \cite{Co46},
Johnson proposed that integral systems should inherently relate to
lattices and, by extension, to crystallographic root systems. He then
arrived to the following definition
\begin{defn}
\label{def:A-basic-systemJohnson}A \emph{basic system of integral
elements} over a real algebra is given by a set of elements such that:
\begin{enumerate}
\item The trace and the norm are integers. Specifically,
\begin{align*}
tr\left(x\right) & =x+\overline{x}\in\mathbb{Z},\\
n\left(x\right) & =x\overline{x}\in\mathbb{Z},
\end{align*}
 and where the conjugation is given by an involution of the algebra.
\item The elements form a subring of the algebra, closed under multiplication,
and comprising a set of invertible unit elements.
\item The elements span a two-, four-, or eight-dimensional lattice embedded
in $\mathbb{R},\mathbb{C},\mathbb{H}$ or $\mathbb{O}$.
\end{enumerate}
\end{defn}

Taking into account Johnson's refinements, the number of viable sets
of integral elements dramatically diminishes, as detailed in Table
\ref{tab:Johnson Integers}).
\begin{table}
\centering{}%
\begin{tabular}{|c|c|c|c|c|c|}
\hline 
\textbf{Name} & \textbf{Alg.} & \textbf{Dim.} & \textbf{Symbol} & \textbf{Unit El.} & \textbf{Lattice}\tabularnewline
\hline 
\hline 
Integers & $\mathbb{R}$ & 1 & $\mathbb{Z}$ & 2 & $A_{1}$\tabularnewline
\hline 
Eisenstein & $\mathbb{C}$ & 2 & $\mathbb{C}_{A_{2}}$ & 3 & $A_{2}$\tabularnewline
\hline 
Gaussian & $\mathbb{C}$ & 2 & $\mathbb{C}_{C_{2}}$ & 4 & $C_{2}$\tabularnewline
\hline 
Hamilton & $\mathbb{H}$ & 4 & $\mathbb{H}_{2C_{2}}$ & 8 & $C_{2}\oplus C_{2}$\tabularnewline
\hline 
Hybrid & $\mathbb{H}$ & 4 & $\mathbb{H}_{2A_{2}}$ & 12 & $A_{2}\oplus A_{2}$\tabularnewline
\hline 
Hurwitz & $\mathbb{H}$ & 4 & $\mathbb{H}_{D_{4}}$ & 24 & $D_{4}$\tabularnewline
\hline 
Cayley-Graves & $\mathbb{O}$ & 8 & $\mathbb{O}_{C_{8}}$ & 16 & $C_{8}$\tabularnewline
\hline 
Comp. Eisenstein & $\mathbb{O}$ & 8 & $\mathbb{O}_{4A_{2}}$ & 24 & $A_{2}\oplus A_{2}\oplus A_{2}\oplus A_{2}$\tabularnewline
\hline 
Coupled Hurwitz & $\mathbb{O}$ & 8 & $\mathbb{O}_{2D_{4}}$ & 48 & $D_{4}\oplus D_{4}$\tabularnewline
\hline 
Coxeter-Dickson & $\mathbb{O}$ & 8 & $\mathbb{O}_{E_{8}}$ & 240 & $E_{8}$\tabularnewline
\hline 
\end{tabular}\caption{\label{tab:Johnson Integers}Summary of all crystallographic sets
of integer elements over division Hurwitz algebra $\mathbb{R},\mathbb{C},\mathbb{H}$
and $\mathbb{O}$. In the first column we indicated the name according
to \cite{Jo13}; then the related algebra in which the integral set
is embedded; the dimension of the algebra; the notational symbol we
introduced; the number of invertible unit elements; their algebraic
structure as abelian group (real and complex case), non-abelian group
(quaternionic case) and as Moufang loop (octonionic case); finally,
in the last column the lattice associated with the integral set.}
\end{table}

\subsection{Complex integral sets}

In order to appreciate the algebraic richness of the octonionic case,
it is instructive to begin with the simpler complex setting. The complex
plane $\mathbb{C}$ admits two classical \emph{crystallographic integral
sets}, each corresponding to a different sublattice of the Gaussian
plane.

The \emph{Gaussian integers} $\mathbb{Z}[i]=\left\{ a+bi:a,b\in\mathbb{Z}\right\} $
form a ring under standard complex multiplication, with norm $n(a+bi)=a^{2}+b^{2}\in\mathbb{Z}$
and trace $tr(a+bi)=2a\in\mathbb{Z}$. The four invertible unit elements
are $\left\{ +1,-1,+i,-i\right\} $, forming the cyclic group $\mathbb{Z}_{4}$.
Their lattice is isometric to $C_{2}$, the square lattice.

The \emph{Eisenstein integers} $\mathbb{Z}[\omega]$, where $\omega=e^{2\pi i/3}=\frac{-1+i\sqrt{3}}{2}$,
constitute a second integral system. Here the norm is $n(a+b\omega)=a^{2}-ab+b^{2}\in\mathbb{Z}$
and the trace is $tr(a+b\omega)=2a-b\in\mathbb{Z}$. The six unit
elements are $\left\{ \pm1,\pm\omega,\pm\omega^{2}\right\} $, forming
the cyclic group $\mathbb{Z}_{6}$, and their lattice is the triangular
lattice $A_{2}$.

\begin{figure}
\begin{centering}
\includegraphics[width=0.85\textwidth]{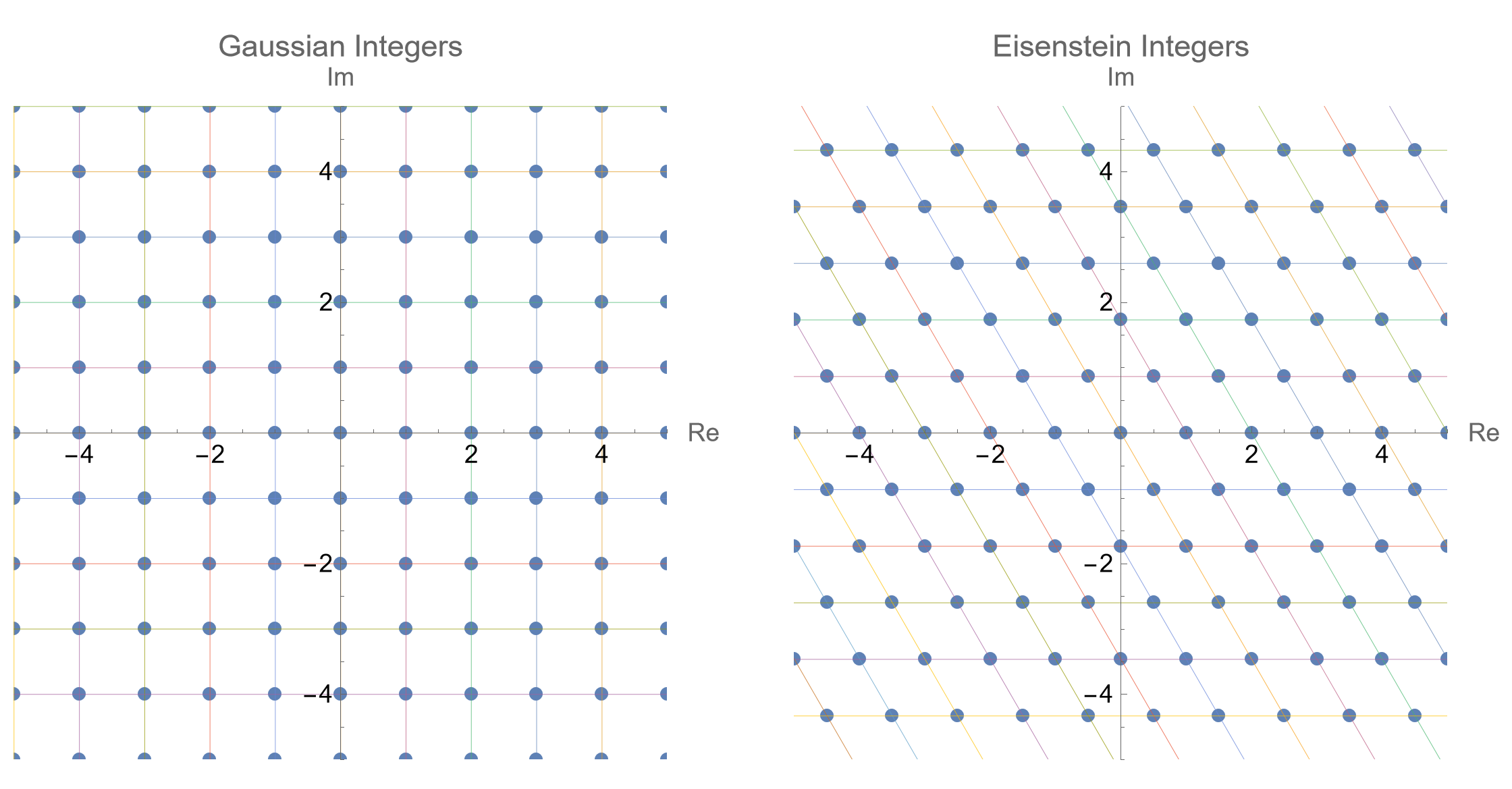}
\par\end{centering}
\caption{\emph{\label{fig:ComplexIntegers}On the left}: the Gaussian integers
$\mathbb{Z}[i]$, forming a square lattice $C_{2}$ in the complex
plane with four unit elements $\left\{ \pm1,\pm i\right\} $. \emph{On
the right}: the Eisenstein integers $\mathbb{Z}[\omega]$, forming
a triangular lattice $A_{2}$ with six unit elements $\left\{ \pm1,\pm\omega,\pm\omega^{2}\right\} $.}
\end{figure}

\begin{rem}
\label{rem:sqrt3-Eisenstein}It is worth noting that $\sqrt{3}$ appears
in the Eisenstein case through the imaginary part of $\omega$. This
is not merely a numerical coincidence: it reflects the underlying $A_{2}$
symmetry of the hexagonal lattice. The same coefficient appears in a useful
rotation presentation of the Okubo construction. As we shall see, however, it
belongs to that presentation and is not an intrinsic obstruction to integral
Okubo orders over $\mathbb{Z}$.
\end{rem}

\subsection{Quaternionic integral sets}

Moving to dimension four, the quaternions $\mathbb{H}$ admit three
crystallographic integral sets, reflecting a richer algebraic structure.

The \emph{Hamilton integers} $\mathbb{H}_{2C_{2}}$ consist of all
quaternions $a_{0}+a_{1}\mathrm{i}+a_{2}\mathrm{j}+a_{3}\mathrm{k}$
with $a_{0},a_{1},a_{2},a_{3}\in\mathbb{Z}$. Their eight unit elements
$\left\{ \pm1,\pm\mathrm{i},\pm\mathrm{j},\pm\mathrm{k}\right\} $
form the quaternion group $Q_{8}$, and the lattice is $C_{2}\oplus C_{2}$.

The \emph{Hurwitz integers} $\mathbb{H}_{D_{4}}$ are the integral
closure of $\mathbb{H}_{2C_{2}}$ in $\mathbb{H}$: they consist of
all quaternions of the form $a_{0}+a_{1}\mathrm{i}+a_{2}\mathrm{j}+a_{3}\mathrm{k}$
where either all $a_{i}\in\mathbb{Z}$ or all $a_{i}\in\mathbb{Z}+\frac{1}{2}$.
They possess 24 unit elements forming the \emph{binary tetrahedral
group} $2T\cong\mathrm{SL}(2,3)$, whose underlying lattice is $D_{4}$.
The 24 unit elements correspond to the vertices of the 24-cell, the
unique regular polytope in dimension four which is its own dual. In
analogy to the octonionic case, the Hurwitz integers \emph{enjoy}
the property of being the maximal order in $\mathbb{H}$ over $\mathbb{Z}$.

The \emph{Hybrid integers} $\mathbb{H}_{2A_{2}}$ are generated by
combining two Eisenstein systems within $\mathbb{H}$ and possess 12
unit elements. Their lattice is $A_{2}\oplus A_{2}$.

\begin{rem}
\label{rem:Hurwitz-maximal}The Hurwitz integers $\mathbb{H}_{D_{4}}$
are, up to isomorphism, the unique maximal order in $\mathbb{H}$ over
$\mathbb{Z}$. This maximality is the quaternionic analogue of the
Coxeter-Dickson octonions being the maximal order in $\mathbb{O}$,
a theme that recurs, with important modifications, in the Okubo setting
where the product is changed while the coefficient ring remains $\mathbb{Z}$.
\end{rem}

\subsection{Octonionic integral sets}

The octonionic case is considerably richer than its lower-dimensional
counterparts. In analogy to the quaternionic case, there are multiple
crystallographic integral sets in $\mathbb{O}$, nested by inclusion
and distinguished by the number of unit elements they contain.

The \emph{Cayley-Graves integers} $\mathbb{O}_{C_{8}}$ are the simplest
octonionic integral set, consisting of octonions with all-integer coefficients
in the standard basis $\left\{ 1,\mathrm{i},\mathrm{j},\mathrm{k},\mathrm{l},\mathrm{il},\mathrm{jl},\mathrm{kl}\right\} $.
They possess 16 unit elements forming the Moufang loop $M_{16}(Q_{8})$,
and their lattice is $C_{8}$.

The \emph{Compounded Eisenstein integers} $\mathbb{O}_{4A_{2}}$ are
obtained by embedding four copies of the Eisenstein integers into $\mathbb{O}$
using compatible quaternionic subalgebras. They possess 24 unit elements
forming the loop $M_{24}(4A_{2})$, and their lattice is $A_{2}\oplus A_{2}\oplus A_{2}\oplus A_{2}$.

The \emph{Coupled Hurwitz integers} $\mathbb{O}_{2D_{4}}$ are obtained
by coupling two copies of the Hurwitz quaternionic integers in a compatible
manner. They possess 48 unit elements forming the loop $M_{48}(2D_{4})$,
and their lattice is $D_{4}\oplus D_{4}$.

Finally, the maximal integral set is the \emph{Coxeter-Dickson octonions}
$\mathbb{O}_{E_{8}}$, to which we now turn. As we will see, this
is the unique octonionic integral system whose lattice is the exceptional
$E_{8}$, and it is the one against which the para-octonionic and Okubo
products will be tested.

For our purposes, the integral set called Coxeter-Dickson octonions,
and denoted here as $\mathbb{O}_{E_{8}}$, is of special importance.
Following \cite{Co46}, we define as the octonionic elements $\mathbb{O}_{E_{8}}$,
those of the form

\begin{equation}
x=a_{0}+a_{1}\text{i}+a_{2}\text{j}+a_{3}\text{k}+a_{4}\text{h}+a_{5}\text{ih}+a_{6}\text{jh}+a_{7}\text{kh},
\end{equation}
 where $\text{h}=\left(\text{i}+\text{j}+\text{k}+\text{l}\right)/2$
and $a_{0},...,a_{7}\in\mathbb{Z}$. 

To see that this is indeed a set of integral elements under octonionic
conjugation and multiplication we have first to show that the norm
and the trace are integers. Indeed, a straightforward calculation
yields to an explicit form for the norm 
\begin{align}
n\left(x\right) & =a_{0}\left(a_{0}-a_{5}-a_{6}-a_{7}\right)+a_{1}\left(a_{1}+a_{4}+a_{6}\right)+\\
 & +a_{2}\left(a_{2}+a_{4}+a_{5}\right)+a_{3}\left(a_{3}+a_{4}\right)+a_{4}^{2}+a_{5}^{2}+a_{6}^{2}+a_{7}^{2},
\end{align}
and for the trace
\begin{equation}
tr\left(x\right)=2a_{0}-a_{5}-a_{6}-a_{7},
\end{equation}
thus showing that $n\left(x\right),tr\left(x\right)\in\mathbb{Z}$.
The unit invertible elements of the set are the 240 elements 

\begin{equation}
\begin{array}{cc}
\pm1,\pm\text{i},\pm\text{j},\pm\text{k}, & \pm\text{l},\pm\text{i\text{l}},\pm\text{j\text{l}},\pm\text{k\text{l}},\\
\frac{1}{2}\left(\pm1\pm\text{j}\pm\text{k}\pm\text{i\text{l}}\right), & \frac{1}{2}\left(\pm\text{i}\pm\text{l}\pm\text{j\text{l}}\pm\text{k\text{l}}\right),\\
\frac{1}{2}\left(\pm1\pm\text{k}\pm\text{i}\pm\text{j\text{l}}\right), & \frac{1}{2}\left(\pm\text{j}\pm\text{l}\pm\text{k\text{l}}\pm\text{i\text{l}}\right),\\
\frac{1}{2}\left(\pm1\pm\text{i}\pm\text{j}\pm\text{k\text{l}}\right), & \frac{1}{2}\left(\pm\text{k}\pm\text{l}\pm\text{i\text{l}}\pm\text{j\text{l}}\right),\\
\frac{1}{2}\left(\pm1\pm\text{i\text{l}}\pm\text{j\text{l}}\pm\text{k\text{l}}\right), & \frac{1}{2}\left(\pm\text{i}\pm\text{j}\pm\text{k}\pm\text{l}\right),\\
\frac{1}{2}\left(\pm1\pm\text{i}\pm\text{l}\pm\text{i\text{l}}\right), & \frac{1}{2}\left(\pm\text{j}\pm\text{k}\pm\text{j\text{l}}\pm\text{k\text{l}}\right),\\
\frac{1}{2}\left(\pm1\pm\text{j}\pm\text{l}\pm\text{j\text{l}}\right), & \frac{1}{2}\left(\pm\text{k}\pm\text{i}\pm\text{k\text{l}}\pm\text{i\text{l}}\right),\\
\frac{1}{2}\left(\pm1\pm\text{k}\pm\text{l}\pm\text{k\text{l}}\right), & \frac{1}{2}\left(\pm\text{i}\pm\text{j}\pm\text{i\text{l}}\pm\text{j\text{l}}\right),
\end{array}\label{eq:InvertibleElements}
\end{equation}
which identify the vertices of the Gosset polytope $4_{21}$. Those
are closed under octonionic multiplication (see Fig. \ref{fig:Fano-Plane}).
In fact, the elements in (\ref{eq:InvertibleElements}) form a well-known
Moufang loop, here denoted as $M_{240}\left(E_{8}\right)$. 

Finally, the isometry between the set of integral octonions $\mathbb{O}_{E_{8}}$
with the octonionic norm $n$ and the $\text{E}_{8}$-lattice is a
well-known result from a straightforward calculation and that can
be explicitly found in \cite[Sec. 9]{CoSm} or \cite{Koc89}. 
\begin{figure}
\begin{centering}
\includegraphics[scale=0.8]{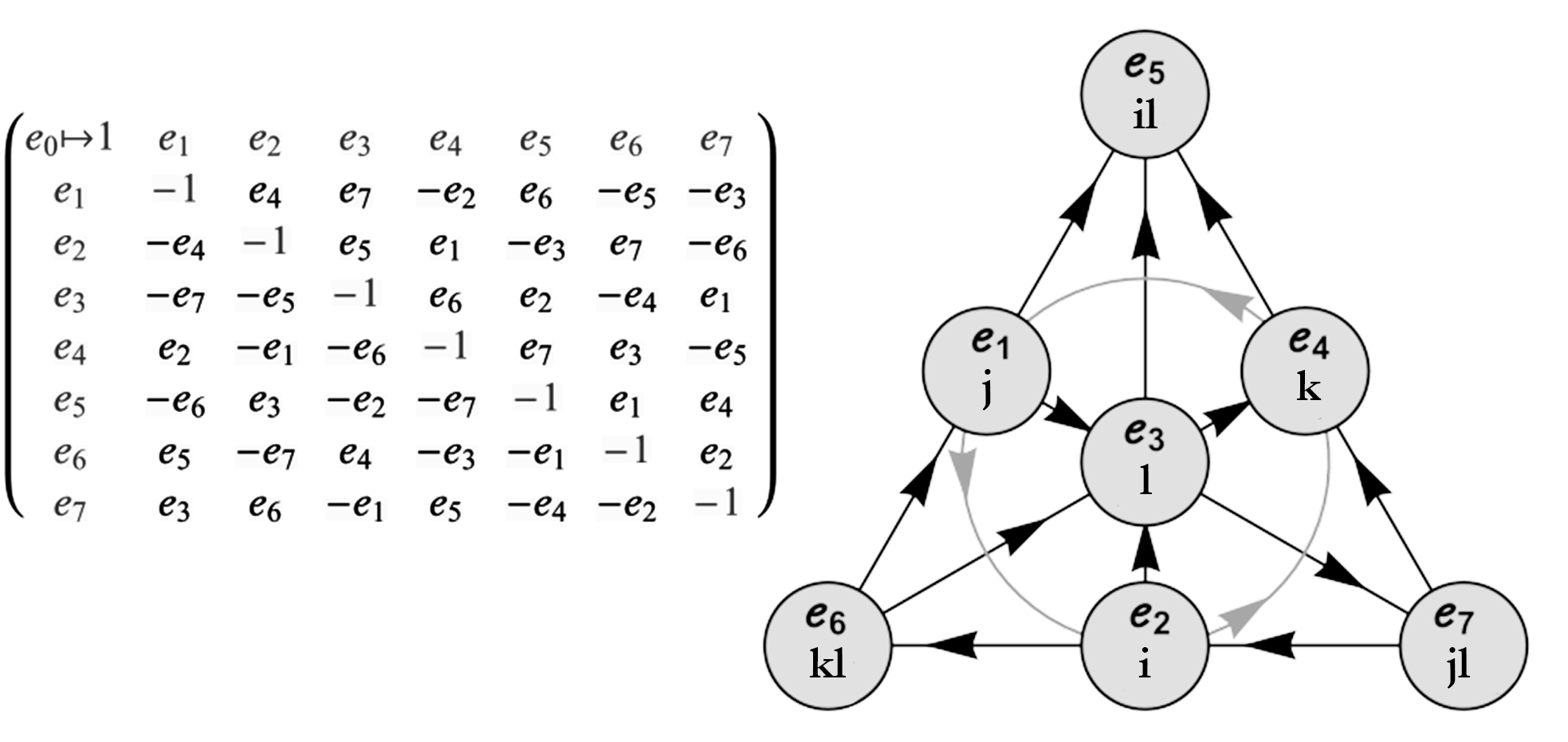}
\par\end{centering}
\caption{\emph{\label{fig:Fano-Plane}On the left}: octonionic multiplication
tables for the basis $\left\{ e_{0}=1,e_{1},e_{2},e_{3},e_{4},e_{5},e_{6},e_{7}\right\} $.
\emph{On the right}: a mnemonic representation on the Fano plane of
the same octonionic multiplication rule with the equivalence with
the Dickson notation $\left\{ 1,\text{i},\text{j},\text{k},\text{l},\text{il},\text{jl},\text{kl}\right\} $
according to \cite{Co46}.}
\end{figure}

\begin{rem}
\label{rem:shell-formula}The shell structure of $\mathbb{O}_{E_{8}}$
is captured by the remarkable formula
\begin{equation}
\left|\mathbb{O}_{E_{8}}(n)\right|=240\sum_{d|n}d^{3},\label{eq:shell-formula}
\end{equation}
where $\left|\mathbb{O}_{E_{8}}(n)\right|$ denotes the number of elements
of norm $n$ in $\mathbb{O}_{E_{8}}$. In particular, the 240 unit elements
($n=1$) are exactly the vertices of the Gosset polytope $4_{21}$.
This formula reflects the deep connection between $E_{8}$ and modular
forms: the right-hand side is proportional to the Eisenstein series $E_{4}$.
\end{rem}

\section{\noun{Octonions, paraoctonions and the real Okubo algebra}}

Lattices are often constructed by taking discrete subsets of unit
norm elements within an algebra and subsequently forming their integer
span. Consequently, possessing a quadratic norm that respects the
multiplicative structure of the algebra results in an intriguing interplay
between the lattice elements and their algebraic properties. Normed
algebras endowed with a norm that respect the multiplicative structure,
i.e., $n\left(x\cdot y\right)=n\left(x\right)n\left(y\right)$, are
called \emph{composition algebras}.

Composition algebras are divided in \emph{unital}, i.e. where exists
a unit element $1$ such that $x\cdot1=1\cdot x=x$, \emph{para-unital},
i.e. where exists an involution $x\longrightarrow\overline{x}$ and
an element called paraunit $\boldsymbol{1}$ such that $x\cdot\boldsymbol{1}=\boldsymbol{1}\cdot x=\overline{x}$,
and, finally, \emph{non-unital}, i.e. intended here for algebras that
possess neither a unit element nor a paraunit element. 
\begin{table}
\begin{centering}
\begin{tabular}{|c|c|c|c|c|c|c|c|c|c|c|c|c|}
\hline 
\textbf{Hurwitz} & \textbf{O.} & \textbf{C.} & \textbf{A.} & \textbf{Alt.} & \textbf{F.} &  & \textbf{p-Hurwitz} & \textbf{O.} & \textbf{C.} & \textbf{A.} & \textbf{Alt.} & \textbf{F.}\tabularnewline
\hline 
\hline 
$\mathbb{R}$ & Yes & Yes & Yes & Yes & Yes &  & $p\mathbb{R}\cong\mathbb{R}$ & Yes & Yes & Yes & Yes & Yes\tabularnewline
\hline 
$\mathbb{C}$, $\mathbb{C}_{s}$ & No & Yes & Yes & Yes & Yes &  & $p\mathbb{C}$, $p\mathbb{C}_{s}$ & No & Yes & No & No & Yes\tabularnewline
\hline 
$\mathbb{H}$,$\mathbb{H}_{s}$ & No & No & Yes & Yes & Yes &  & $p\mathbb{H}$,$p\mathbb{H}_{s}$ & No & No & No & No & Yes\tabularnewline
\hline 
$\mathbb{O}$,$\mathbb{O}_{s}$ & No & No & No & Yes & Yes &  & $p\mathbb{O}$,$p\mathbb{O}_{s}$ & No & No & No & No & Yes\tabularnewline
\hline 
\end{tabular}
\par\end{centering}
\centering{}{\small\bigskip{}
}\caption{\emph{\label{tab:Hurwitz-para-Hurwitz-1}On the left,} we have summarized
the algebraic properties, i.e. totally ordered (O), commutative (C),
associative (A), alternative (Alt), flexible (F), of all Hurwitz algebras,
namely $\mathbb{R},\mathbb{C},\mathbb{H}$ and $\mathbb{O}$ along
with their split counterparts $\mathbb{C}_{s},\mathbb{H}_{s},\mathbb{O}_{s}$.
\emph{On the right}, we have summarized the algebraic properties of
all para-Hurwitz algebras, namely $p\mathbb{R},p\mathbb{C},p\mathbb{H}$
and $p\mathbb{O}$ accompanied by their split counterparts $p\mathbb{C}_{s},p\mathbb{H}_{s},p\mathbb{O}_{s}$.}
\end{table}

In fact, as a consequence of the Generalized Hurwitz Theorem \cite{ElDuque Comp},
over the real numbers, the finite-dimensional composition algebras fall into exactly sixteen isomorphism classes: seven
are unital and called \emph{Hurwitz algebras}, comprising the division
algebras $\mathbb{R},\mathbb{C},\mathbb{H},\mathbb{O}$ along with
their split companions $\mathbb{C}_{s},\mathbb{H}_{s},\mathbb{O}_{s}$;
another seven are paraunital, closely related to the Hurwitz algebras
and termed \emph{para-Hurwitz algebras} (see Table \ref{tab:Hurwitz-para-Hurwitz-1});
finally, there are two composition algebras, one division and one
split, that are both non-unital and 8-dimensional, known as the \emph{Okubo
algebras} $\mathcal{O}$ and $\mathcal{O}_{s}$. 

In this section we review the three division and composition algebras
of dimension eight: the Hurwitz algebra of the octonions $\mathbb{O}$,
the para-Hurwitz algebra of para-octonions $p\mathbb{O}$ and, finally,
the Okubo algebra $\mathcal{O}$. The main objective of this section
is summarized in Table \ref{tab:Oku-Para-Octo} that synoptically
illustrates the relationships between the product of the three algebras.
The reason for limiting ourselves to the division algebras and not
consider their split companions it will be clear in the next section,
where we will require the closure under multiplication for
the elements of unit norm.

An important final remark: these three algebras are deeply linked
one another and it might be tempting to consider them equivalent in
some sort of sense. Nevertheless, it is crucial to note that although
is possible to switch from one algebra to another by altering the
product definition, none of these algebras are isomorphic to the others.
Specifically, the octonions $\mathbb{O}$ are alternative and unital,
para-octonions $p\mathbb{O}$ are neither alternative nor unital but do
have a para-unit, while the Okubo algebra $\mathcal{O}$ is non-alternative
and only has idempotents elements. It is also worth highlighting that
the Okubo algebra $\mathcal{O}$ is the least structured among these
algebras (for a summary of the property of these algebras see Table
\ref{tab:Synoptic-table-of}).
\begin{table}
\centering{}%
\begin{tabular}{|c|c|c|c|}
\hline 
Property & $\mathbb{O}$ & $p\mathbb{O}$ & $\mathcal{O}$\tabularnewline
\hline 
\hline 
Unital & Yes & No & No\tabularnewline
\hline 
Paraunital & Yes & Yes & No\tabularnewline
\hline 
Alternative & Yes & No & No\tabularnewline
\hline 
Flexible & Yes & Yes & Yes\tabularnewline
\hline 
Composition & Yes & Yes & Yes\tabularnewline
\hline 
Automorphism & $\text{G}_{2}$ & $\text{G}_{2}$ & $\text{SU}\left(3\right)$\tabularnewline
\hline 
\end{tabular}\caption{\label{tab:Synoptic-table-of}Synoptic table of the algebraic properties
of octonions $\mathbb{O}$, paraoctonions $p\mathbb{O}$ and the real
Okubo algebra $\mathcal{O}$.}
\end{table}

\subsection{The algebra of octonions}

The algebra of octonions $\mathbb{O}$ is the only division Hurwitz
algebra with a dimension of eight. We define the composition algebra
of octonion $\left(\mathbb{O},\cdot,n\right)$ as the eight dimensional
real vector space with basis $\left\{ \text{e}_{0}=1,\text{e}_{1},...,\text{e}_{7}\right\} $
with a bilinear product encoded through the \emph{Fano plane} and
explained in Fig. \ref{fig:Fano-Plane}. The resulting algebra is
non-associative, non-commutative, but alternative and thus flexible.
Given an element $x\in\mathbb{O}$ with decomposition 
\begin{equation}
x=x_{0}+\stackrel[k=1]{7}{\sum}x_{k}\text{e}_{k},
\end{equation}
the norm $n$ is the obvious Euclidean one defined by 
\begin{equation}
n\left(x\right)=x_{0}^{2}+x_{1}^{2}+x_{2}^{2}+x_{3}^{2}+x_{4}^{2}+x_{5}^{2}+x_{6}^{2}+x_{7}^{2},\label{eq:octonionic norm}
\end{equation}
 for which the conjugation results 
\begin{equation}
\overline{x}=x_{0}-\stackrel[k=1]{7}{\sum}x_{k}i_{k},
\end{equation}
 and therefore 
\begin{equation}
n\left(x\right)=\overline{x}\cdot x,\label{eq:n(x)=00003Dxcx}
\end{equation}
as it happens for every Hurwitz algebra. Then a look at (\ref{eq:octonionic norm})
shows that $n\left(x\right)=0$ if and only if $x=0$ and thus the
inverse of a non-zero element of the octonions is easily found as
\begin{equation}
x^{-1}=\frac{\overline{x}}{n\left(x\right)}.
\end{equation}
 Also, from (\ref{eq:n(x)=00003Dxcx}) we have that the octonionic
inner product is given by 
\begin{align}
\left\langle x,y\right\rangle  & =x\overline{y}+y\overline{x},\label{eq:octonionic inner}
\end{align}
so that $\left\langle x,x\right\rangle =2n\left(x\right)$.

\subsection{\label{subsec:Conjugation-and-the}The algebra of para-octonions}

In unital composition algebras, as noted earlier, there exists a canonical
involution, an order-two antihomomorphism known as \emph{conjugation}.
This can be defined using the orthogonal projection of the unit element
as
\begin{equation}
x\mapsto\overline{x}=\left\langle x,1\right\rangle 1-x.\label{eq:coniugazione}
\end{equation}
This canonical involution has the distinctive property of being an
antihomomorphism with respect to the product, i.e., $\overline{x\cdot y}=\overline{y}\cdot\overline{x},$
and the basic property with the norm of $x\cdot\overline{x}=n\left(x\right)1$. 

Given the order-two antihomomorphism of the conjugation over the Hurwitz
algebra of octonions $\left(\mathbb{O},\cdot,n\right)$ we can now
obtain a para-Hurwitz algebra defining a new product
\begin{equation}
x\bullet y=\overline{x}\cdot\overline{y},
\end{equation}
 for every $x,y\in\mathbb{O}$. The new algebra $\left(\mathbb{O},\bullet,n\right)$
is again a composition algebra, in fact a para-Hurwitz algebra, called
para-octonions and denoted with $p\mathbb{O}$. As expected the algebra
does not have a unit but only a para-unit, i.e. $1\in p\mathbb{O}$
such that $1\bullet x=\overline{x}.$ Finally, it is worth noting
that also the algebra of para-octonions is a division algebra since
if 
\begin{equation}
x\bullet y=\overline{x}\cdot\overline{y}=0,
\end{equation}
then either $\overline{x}$ or $\overline{y}$ are zero and thus implying
that either $x$ or $y$ are zero.

\subsection{The Okubo algebra}

The Petersson construction depends not only on the real automorphism, but also
on its position with respect to a chosen integral order. This distinction is
invisible over $\mathbb{R}$ and becomes essential over $\mathbb{Z}$. We begin
with the rotation presentation used in the our previous works such as \cite{CMZ24-MinimalCayley,CMZ25-Collineations,MCZ25-OkuboPhysics}. To
avoid confusion, we denote it by $\rho$:
\begin{equation}
\begin{array}{cc}
\rho\left(\text{e}_{k}\right) & =\text{e}_{k},k=0,1,3,7,\\
\rho\left(\text{e}_{2}\right) & =-\frac{1}{2}\left(\text{e}_{2}-\sqrt{3}\text{e}_{5}\right),\\
\rho\left(\text{e}_{5}\right) & =-\frac{1}{2}\left(\text{e}_{5}+\sqrt{3}\text{e}_{2}\right),\\
\rho\left(\text{e}_{4}\right) & =-\frac{1}{2}\left(\text{e}_{4}-\sqrt{3}\text{e}_{6}\right),\\
\rho\left(\text{e}_{6}\right) & =-\frac{1}{2}\left(\text{e}_{6}+\sqrt{3}\text{e}_{4}\right).
\end{array}\label{eq:old-rho}
\end{equation}
This is an order-three automorphism of $\mathbb{O}$. It fixes the quaternion
subalgebra generated by $\{1,\text{e}_{1},\text{e}_{3},\text{e}_{7}\}$. This
subalgebra contains
$w_{\rho}=(-1+\sqrt{3}\text{e}_{1})/2$, with
$w_{\rho}^{2}+w_{\rho}+1=0$. Therefore
\cite[Theorem~5.1]{ElduqueOrder3} shows that its Petersson isotope is a real
Okubo algebra. There is no algebraic error in this construction. The problem is arithmetic: already
$\rho(\text{e}_{2})$ does not belong to the Cayley--Graves order. Moreover, an
explicit product of Coxeter--Dickson basis elements will show that its
Petersson product does not preserve $\mathbb{O}_{E_{8}}$ either.

The two real planes generated by $\{\text{e}_{2},\text{e}_{5}\}$ and
$\{\text{e}_{4},\text{e}_{6}\}$ explain the displayed coefficients. On both
planes $\rho$ acts as a rotation by $2\pi/3$, with the same matrix as
multiplication by the primitive cubic root
$\omega=-\frac{1}{2}+\frac{\sqrt{3}}{2}\mathrm{i}$.

For integral questions we use instead the following signed permutation of the
Fano basis:
\begin{equation}
\begin{array}{lll}
\tau(1)=1, & \tau(\text{e}_{3})=\text{e}_{3}, & \\
\tau(\text{e}_{1})=\text{e}_{2}, &
\tau(\text{e}_{2})=-\text{e}_{4}, &
\tau(\text{e}_{4})=-\text{e}_{1},\\
\tau(\text{e}_{5})=-\text{e}_{7}, &
\tau(\text{e}_{7})=\text{e}_{6}, &
\tau(\text{e}_{6})=-\text{e}_{5}.
\end{array}\label{eq:Tau(Octonions)}
\end{equation}
Direct inspection of the seven oriented lines of the Fano plane gives
$\tau(x\cdot y)=\tau(x)\cdot\tau(y)$, while the two signed three-cycles give
$\tau^{3}=\operatorname{id}$. Thus $\tau\in\mathrm{G}_{2}$ is an
order-three automorphism defined over $\mathbb{Z}$.

It remains to determine which symmetric composition algebra is obtained. Set
\begin{equation}
a=-\text{e}_{1}-\text{e}_{2}+\text{e}_{4},\qquad
b=-\text{e}_{5}+\text{e}_{6}+\text{e}_{7}.
\end{equation}
The fixed subalgebra is
\begin{equation}
\operatorname{Fix}(\tau)=
\operatorname{span}_{\mathbb{R}}\{1,\text{e}_{3},a,b\}.
\end{equation}
Indeed, the Fano multiplication gives
\begin{equation}
a^{2}=b^{2}=-3,\qquad ab=3\text{e}_{3}=-ba,
\end{equation}
so this fixed subalgebra is isomorphic to the quaternions. In addition,
$w=(-1+a)/2$ satisfies $w^{2}+w+1=0$. By the classification of Petersson
algebras associated with order-three automorphisms
\cite[Theorem~5.1]{ElduqueOrder3}, the corresponding isotope is an Okubo
algebra, and not a para-octonionic algebra. Since the norm is positive definite, it
is the compact real division Okubo algebra.

We can now define its Petersson product by
\begin{equation}
x*y=\tau\left(\overline{x}\right)\cdot\tau^{2}\left(\overline{y}\right),
\end{equation}
for every $x,y\in\mathbb{O}$. Note that the element $1$ is an idempotent, since
$1*1=1$, but it is not a unit or a paraunit. It is worth noting that this Okubo
algebra is a division algebra: if
\begin{equation}
x*y=\tau\left(\overline{x}\right)\cdot\tau^{2}\left(\overline{y}\right)=0,
\end{equation}
then either $\tau\left(\overline{x}\right)$ or $\tau^{2}\left(\overline{y}\right)$
are zero and since $\tau$ is a homomorphism, either $x$ or $y$
are zero.

Unlike the case of the para-octonions it is worth presenting here
an independent way of realising the Okubo algebra. In fact, the following
realisation was the one independently found by Okubo ten years after
Petersson. Following \cite{Okubo 1978} and \cite{Elduque Myung 90},
we define the real Okubo Algebra $\mathcal{O}$ as the set of three
by three Hermitian traceless matrices over the complex numbers $\mathbb{C}$
with the following bilinear product 
\begin{equation}
x*y=\mu\cdot xy+\overline{\mu}\cdot yx-\frac{1}{3}\text{Tr}\left(xy\right)I_{3},\label{eq:product Ok}
\end{equation}
where $\mu=\nicefrac{1}{6}\left(3+\text{i}\sqrt{3}\right)$ and the
juxtaposition is the ordinary associative product between matrices.
It is worth noting that (\ref{eq:product Ok}) can be seen as a modification
of the Jordanian product. Indeed, setting $\mu=\nicefrac{1}{2}$ and
neglecting the last term, we retrieve the usual Jordan product over
Hermitian traceless matrices, i.e.
\begin{equation}
x\circ y=\frac{1}{2}xy+\frac{1}{2}yx.
\end{equation}
Nevertheless, Hermitian traceless matrices are not closed under such
product, thus requiring the additional term $-\nicefrac{1}{3}\text{Tr}\left(xy\right)$
for the closure of the algebra. Indeed, setting in (\ref{eq:product Ok})
$\text{Im}\mu=0$, one retrieves from the traceless part of the exceptional
Jordan algebra $\mathfrak{J}_{3}\left(\mathbb{C}\right)$, whose derivation
Lie algebra is $\mathfrak{su}\left(3\right)$.

Analyzing (\ref{eq:product Ok}), it becomes evident that the resulting
algebra is neither unital, associative, nor alternative. Nonetheless,
$\mathcal{O}$ is a\emph{ flexible }algebra, i.e. 
\begin{equation}
x*\left(y*x\right)=\left(x*y\right)*x,
\end{equation}
which will turn out to be an even more useful property than alternativity
in the definition of the projective plane. Even though the Okubo algebra
is not unital, it does have idempotents, i.e. $e*e=e$, such as 
\begin{equation}
e=\left(\begin{array}{ccc}
2 & 0 & 0\\
0 & -1 & 0\\
0 & 0 & -1
\end{array}\right),\label{eq:idemp}
\end{equation}
that together with
\begin{equation}
\begin{array}{ccc}
\text{e}_{1}=\sqrt{3}\left(\begin{array}{ccc}
0 & 1 & 0\\
1 & 0 & 0\\
0 & 0 & 0
\end{array}\right), &  & \text{e}_{2}=\sqrt{3}\left(\begin{array}{ccc}
0 & 0 & 1\\
0 & 0 & 0\\
1 & 0 & 0
\end{array}\right),\\
\text{e}_{3}=\sqrt{3}\left(\begin{array}{ccc}
0 & 0 & 0\\
0 & 0 & 1\\
0 & 1 & 0
\end{array}\right), &  & \text{e}_{4}=\sqrt{3}\left(\begin{array}{ccc}
1 & 0 & 0\\
0 & -1 & 0\\
0 & 0 & 0
\end{array}\right),\\
\text{e}_{5}=\sqrt{3}\left(\begin{array}{ccc}
0 & -i & 0\\
i & 0 & 0\\
0 & 0 & 0
\end{array}\right), &  & \text{e}_{6}=\sqrt{3}\left(\begin{array}{ccc}
0 & 0 & -i\\
0 & 0 & 0\\
i & 0 & 0
\end{array}\right),\\
\text{e}_{7}=\sqrt{3}\left(\begin{array}{ccc}
0 & 0 & 0\\
0 & 0 & -i\\
0 & i & 0
\end{array}\right),
\end{array}\label{eq:definizione i ottonioniche}
\end{equation}
form a basis for $\mathcal{O}$ that has real dimension $8$. It is
worth noting that the choice of the idempotent $e$ as in (\ref{eq:idemp})
does not yield to any loss of generality for the subsequent development
of our work since all idempotents form a single orbit under the full
automorphism group $\mathrm{SU}(3)$ (see \cite[Thm.\ 20]{ElduQueAut});
the cyclic subgroup $\langle\tau\rangle\cong\mathbb{Z}_{3}$ alone does
not suffice to connect all idempotents. The choice of this special
basis is useful for the matrix realization, but its orthogonality with respect
to the norm in (\ref{eq:Norm-Ok}) is a separate Gram-matrix check and will not
be used as an arithmetic input below. Through a special bijective map between
Okubo algebra and octonions the elements of the basis
$\left\{ e,\text{e}_{1},...,\text{e}_{7}\right\} $ correspond to the octonionic
ones previously defined. 

In this realisation a direct way of defining the quadratic form $n$
over Okubo algebra, is through the following 
\begin{equation}
n\left(x\right)=\frac{1}{6}\text{Tr}\left(x^{2}\right),\label{eq:Norm-Ok}
\end{equation}
for every $x\in\mathcal{O}$. It is straightforward to see that the
\emph{norm} has signature $(8,0)$ and is a composition norm over
the Okubo algebra; the associated symmetric bilinear form $\langle x,y\rangle$
is associative in the sense $\langle x\cdot z,y\rangle=\langle x,z\cdot y\rangle$.

\subsection{Okubo algebra, octonions and para-octonions}

An important feature of the Okubo algebra $\mathcal{O}$ is its interplay
with the algebra of octonions $\mathbb{O}$. Indeed, octonions and
the Okubo algebra are linked one another in such a way that we can
easily pass from one to the other simply changing the definition of
the bilinear product over the vector space of the algebra. Let us
consider the Kaplansky's trick we introduced earlier and let us define
a new product over the Okubo algebra $\mathcal{O}$ as
\begin{equation}
x\cdot y=\left(e*x\right)*\left(y*e\right),
\end{equation}
where $x,y\in\mathcal{O}$ and $e$ is an idempotent of $\mathcal{O}$.
Given that $e*e=e$ and $n\left(e\right)=1$, the element $e$ acts
as a left and right identity, i.e. 
\begin{align}
x\cdot e & =e*x*e=n\left(e\right)x=x,\\
e\cdot x & =e*x*e=n\left(e\right)x=x.
\end{align}
Moreover, since Okubo algebra is a composition algebra, the same norm
$n$ enjoys the following relation 
\begin{equation}
n\left(x\cdot y\right)=n\left(\left(e*x\right)*\left(y*e\right)\right)=n\left(x\right)n\left(y\right),
\end{equation}
which means that $\left(\mathcal{O},\cdot,n\right)$ is a unital composition
algebra of real dimension $8$. Since it is also a division algebra,
then it must be isomorphic to that of octonions $\mathbb{O}$ as noted
by Okubo himself \cite{Okubo 1978,Okubo 78c}. On the other hand,
al already noticed, if we consider the order three automorphism of
the octonions in (\ref{eq:Tau(Octonions)}), the Okubo algebra is
then realised as a Petersson algebra from the octonions setting
\begin{align}
x*y & =\tau\left(\overline{x}\right)\cdot\tau^{2}\left(\overline{y}\right).
\end{align}
Note that (\ref{eq:Tau(Octonions)}) is formulated assuming the knowledge
of the octonionic product. Reading the same maps as Okubic maps we
then have the notable relation, i.e.

\begin{align}
\overline{x} & =\left(\left(x*e\right)*e\right)*e,\\
\tau\left(x\right) & =\left(\left(\left(x*e\right)*e\right)*e\right)*e,
\end{align}
so that, in fact, the two maps are linked one another, i.e.

\begin{align}
\tau\left(x\right) & =\overline{x}*e\\
\overline{x} & =\tau\left(e*x\right).
\end{align}
While these maps are intertwined, it is important to highlight their
distinct impacts on the algebra's structure. While $\tau$ is an automorphism
for both Okubo algebra $\mathcal{O}$ and octonions $\mathbb{O}$,
$\overline{x}$ do not respect the algebraic structure of the Okubo
algebra $\mathcal{O}$, since it is not an automorphism nor an anti-automorphism
with respect to the Okubo product, while it is an anti-homomorphism
over octonions $\mathbb{O}$.

The scenario with para-octonions, $p\mathbb{O}$ is more straightforward.
By definition, para-octonions are obtainable from octonions $\mathbb{O}$
through
\begin{equation}
x\bullet y=\overline{x}\cdot\overline{y},
\end{equation}
while, on the other hand, octonions $\mathbb{O}$ are obtainable from
para-octonions $p\mathbb{O}$ through the aid of the para-unit $1\in p\mathbb{O}$,
such that
\begin{align}
x\cdot y & =\left(1\bullet x\right)\bullet\left(y\bullet1\right)\\
 & =\overline{x}\bullet\overline{y}=x\cdot y.
\end{align}
The new algebra $\left(p\mathbb{O},\cdot,n\right)$ is again an eight-dimensional
composition algebra which is also unital and division and thus, for
Hurwitz theorem, isomorphic to that of octonions $\mathbb{O}$. Moreover,
since $\tau\left(\overline{x}\right)=\overline{\tau\left(x\right)}$,
we also have that the Okubic algebra is obtainable from the para-Hurwitz
algebra with the introduction of a Petersson-like product, i.e. 
\begin{equation}
x*y=\tau\left(x\right)\bullet\tau^{2}\left(y\right).
\end{equation}
We thus have that all algebras are obtainable one from the other as
summarized in Table \ref{tab:Oku-Para-Octo}. 
\begin{table}
\centering{}%
\begin{tabular}{|c|c|c|c|}
\hline 
Algebra & $\left(\mathcal{O},*,n\right)$ & $\left(p\mathbb{O},\bullet,n\right)$ & $\left(\mathbb{O},\cdot,n\right)$\tabularnewline
\hline 
\hline 
$x*y$ & $x*y$ & $\tau\left(x\right)\bullet\tau^{2}\left(y\right)$ & $\tau\left(\overline{x}\right)\cdot\tau^{2}\left(\overline{y}\right)$\tabularnewline
\hline 
$x\bullet y$ & $\tau^{2}\left(x\right)*\tau\left(y\right)$ & $x\bullet y$ & $\overline{x}\cdot\overline{y}$\tabularnewline
\hline 
$x\cdot y$ & $\left(e*x\right)*\left(y*e\right)$ & $\left(\boldsymbol{1}\bullet x\right)\bullet\left(y\bullet\boldsymbol{1}\right)$ & $x\cdot y$\tabularnewline
\hline 
\end{tabular}\caption{\label{tab:Oku-Para-Octo}In this table we see how to obtain the Okubo
product $*,$ the para-octonionic product $\bullet$ and the octonionic
product $\cdot$ from Okubo algebra $\left(\mathcal{O},*,n\right)$,
para-octonions $\left(p\mathbb{O},\bullet,n\right)$ and octonions
$\left(\mathbb{O},\cdot,n\right)$ respectively}
\end{table}
 Nonetheless, it is vital to note that while transitioning from one
algebra to another is feasible, these algebras are not isomorphic.
For example, while the octonions $\mathbb{O}$ are alternative and
unital, para-octonions $p\mathbb{O}$ are neither alternative nor unital
but do have a para-unit. In contrast, the Okubo algebra $\mathcal{O}$
is non-alternative and only contains idempotent elements.

\section{\noun{Integral systems for the three products}}
\label{sec:integral-systems}

The unit in Johnson's definition must be treated with care when the product is
changed. For the para-octonions the element $1$ is a paraunit, while for the
Petersson model of the Okubo algebra it is an idempotent. The norm and the
Euclidean lattice, on the other hand, do not change. This suggests the following
simple adaptation.

\begin{defn}
\label{def:A-basic-systemNonUnital}Let $A$ be an eight-dimensional real
composition algebra with product $\circ$, norm $n$, and a chosen nonzero
idempotent $e$. A \emph{$\mathbb{Z}$-integral system relative to $e$} is a free
$\mathbb{Z}$-module $L\subset A$ of rank eight such that:
\begin{enumerate}
\item $e\in L$ and $L$ is closed under $\circ$;
\item $n(x)$ and $\operatorname{tr}_{e}(x)=\langle x,e\rangle$ belong to
$\mathbb{Z}$ for every $x\in L$;
\item $L$ is a discrete Euclidean lattice for the bilinear form associated
with $n$.
\end{enumerate}
\end{defn}

The following observation reduces the closure problem to the arithmetic action
of the automorphism.

\begin{prop}[Closure criterion]
\label{prop:closure-criterion}Let $L$ be an octonionic $\mathbb{Z}$-order which
is stable under conjugation, and let $\sigma\in\operatorname{Aut}(\mathbb{O})$
satisfy $\sigma(L)=L$. Then $L$ is closed under both
\begin{equation}
x\bullet y=\overline{x}\cdot\overline{y},\qquad
x*_{\sigma}y=\sigma(\overline{x})\cdot\sigma^{2}(\overline{y}).
\end{equation}
If $1\in L$, the norm and relative trace conditions are the same as for the
original octonionic order.
\end{prop}

\begin{proof}
For $x,y\in L$, conjugation gives $\overline{x},\overline{y}\in L$. The
assumption on $\sigma$ then gives
$\sigma(\overline{x}),\sigma^{2}(\overline{y})\in L$, and octonionic closure
proves both assertions. All three products have the same composition norm.
Moreover, relative to the idempotent $1$, the trace is
$\operatorname{tr}_{1}(x)=\langle x,1\rangle$, which is the usual integral
trace functional on the orders considered below.
\end{proof}

\subsection{Why the old rotation does not work integrally}

Let $*_{\rho}$ denote the Petersson product associated with
(\ref{eq:old-rho}). Since $\rho(\text{e}_{2})$ has nonintegral coordinates,
the Cayley--Graves order is not $\rho$-stable. The same problem occurs for the
Coxeter--Dickson order, although it is less immediate.

Write $\text{h}=(\text{e}_{1}+\text{e}_{2}+\text{e}_{3}+\text{e}_{4})/2$ and
use the Coxeter--Dickson basis
\begin{equation}
\label{eq:coxeter-dickson-basis}
\begin{array}{llll}
b_{0}=1,&b_{1}=\text{e}_{1},&b_{2}=\text{e}_{2},&b_{3}=\text{e}_{3},\\
b_{4}=\text{h},&b_{5}=\text{e}_{1}\text{h},&
b_{6}=\text{e}_{2}\text{h},&b_{7}=\text{e}_{3}\text{h}.
\end{array}
\end{equation}
An exact calculation gives
\begin{align}
\label{eq:okubo-counterexample}
b_{0}*_{\rho}b_{2}={}&-\frac{3}{2}\sqrt{3}\,b_{0}
+\frac{1}{2}\sqrt{3}\,b_{1}
+\left(\frac{1}{2}-\frac{1}{2}\sqrt{3}\right)b_{2}\nonumber\\
&-\sqrt{3}\,b_{5}-\sqrt{3}\,b_{6}-\sqrt{3}\,b_{7}.
\end{align}
Thus $L_{E}*_{\rho}L_{E}\nsubseteq L_{E}$. This is why the old automorphism
does not work for the integral problem.

\begin{rem}
The conclusion is deliberately limited to the chosen presentation. The
automorphism $\rho$ is valid and $*_{\rho}$ gives rise to the compact real Okubo
algebra. Extending coefficients to $\mathbb{Z}[\sqrt{3}]$ or rescaling a basis
may produce orders adapted to $\rho$, but this does not prove that
$\sqrt{3}$ is intrinsic to Okubo integrality. The new automorphism
(\ref{eq:Tau(Octonions)}) removes this presentation-dependent obstruction.
\end{rem}

\subsection{Four orders preserved by the integral automorphism}

We now choose concrete representatives of the four octonionic integral systems
introduced in Sec.~1. The Cayley--Graves order is
\begin{equation}
L_{C}=\mathbb{Z}\{1,\text{e}_{1},\ldots,\text{e}_{7}\}.
\end{equation}
A compatible compounded Eisenstein order $L_{A}$ has basis
\begin{align}
p_{0}&=1,&
p_{1}&=\frac{-1+\text{e}_{5}-\text{e}_{6}-\text{e}_{7}}{2},\nonumber\\
p_{2}&=\frac{-\text{e}_{2}+\text{e}_{4}+\text{e}_{5}+\text{e}_{6}}{2},&
p_{3}&=\frac{-\text{e}_{1}+\text{e}_{2}-\text{e}_{3}-\text{e}_{4}}{2},\nonumber\\
p_{4}&=\frac{\text{e}_{2}+\text{e}_{3}+\text{e}_{5}+\text{e}_{7}}{2},&
p_{5}&=\frac{\text{e}_{1}-\text{e}_{2}-\text{e}_{3}-\text{e}_{4}}{2},\nonumber\\
p_{6}&=\frac{\text{e}_{1}+\text{e}_{2}+\text{e}_{6}-\text{e}_{7}}{2},&
p_{7}&=\frac{-\text{e}_{1}-\text{e}_{2}+\text{e}_{3}-\text{e}_{4}}{2}.
\label{eq:A2-four-basis}
\end{align}
A compatible coupled Hurwitz order $L_{D}$ has basis
\begin{equation}
\label{eq:D4-two-basis}
\begin{array}{llll}
d_{0}=1,&d_{1}=\text{e}_{1},&d_{2}=\text{e}_{2},&
d_{3}=\dfrac{1+\text{e}_{1}+\text{e}_{2}+\text{e}_{4}}{2},\\
d_{4}=\text{e}_{3},&d_{5}=\text{e}_{1}\text{e}_{3},&
d_{6}=\text{e}_{2}\text{e}_{3},&d_{7}=d_{3}\text{e}_{3}.
\end{array}
\end{equation}
Finally, $L_{E}=\mathbb{O}_{E_{8}}$ is the Coxeter--Dickson order with basis
(\ref{eq:coxeter-dickson-basis}). Their metric lattices are respectively
$C_{8}$, $A_{2}^{4}$, $D_{4}^{2}$, and $E_{8}$.

\begin{prop}[Preservation of the four integral systems]
\label{prop:four-integral-systems}The four orders $L_{C}$, $L_{A}$, $L_{D}$,
and $L_{E}$ are stable under the automorphism $\tau$ in
(\ref{eq:Tau(Octonions)}). Consequently, all four are closed under the
para-octonionic product and under the Okubo product
\begin{equation}
\label{eq:petersson-okubo-product}
x*y=\tau(\overline{x})\cdot\tau^{2}(\overline{y}).
\end{equation}
They give $\mathbb{Z}$-integral systems for the three composition algebras.
\end{prop}

\begin{proof}
The displayed bases are closed under octonionic multiplication and conjugation;
this is the usual order property. For $L_{C}$, stability under $\tau$ follows
immediately because (\ref{eq:Tau(Octonions)}) is a signed permutation.
For the other three orders it can be checked without changing coefficient
rings. On the $A_{2}^{4}$ basis one obtains
\begin{equation}
\begin{array}{llll}
\tau(p_{0})=p_{0},&\tau(p_{1})=p_{1},&
\tau(p_{2})=-p_{4}-p_{5},&\tau(p_{3})=p_{5},\\
\tau(p_{4})=p_{6}+p_{7},&\tau(p_{5})=-p_{7},&
\tau(p_{6})=-p_{2},&\tau(p_{7})=-p_{3}.
\end{array}
\end{equation}
On the $D_{4}^{2}$ basis the action is
\begin{equation}
\begin{array}{ll}
\tau(d_{0})=d_{0},\quad \tau(d_{1})=d_{2},&
\tau(d_{2})=d_{0}+d_{1}+d_{2}-2d_{3},\\
\tau(d_{3})=d_{0}+d_{2}-d_{3},\quad \tau(d_{4})=d_{4},&
\tau(d_{5})=d_{6},\\
\tau(d_{6})=d_{4}+d_{5}+d_{6}-2d_{7},&
\tau(d_{7})=d_{4}+d_{6}-d_{7}.
\end{array}
\end{equation}
Finally, on the Coxeter--Dickson basis one has
\begin{align}
\tau(b_{0})&=b_{0},&\tau(b_{1})&=b_{2},&
\tau(b_{2})&=b_{1}+b_{2}+b_{3}-2b_{4},\nonumber\\
\tau(b_{3})&=b_{3},&\tau(b_{4})&=b_{2}+b_{3}-b_{4},&
\tau(b_{5})&=-2b_{1}-b_{2}-b_{3}+2b_{4}+b_{6},\nonumber\\
\tau(b_{6})&=b_{0}+b_{2}+b_{5}+b_{6}+b_{7},\nonumber\\
\tau(b_{7})&=-2b_{0}+2b_{1}+b_{2}+b_{3}\nonumber\\
&\quad -2b_{4}-2b_{6}-b_{7}.
\label{eq:tau-on-E8-basis}
\end{align}
All coefficients are integers. Since $\tau^{3}=\operatorname{id}$, these
inclusions are equalities. Proposition~\ref{prop:closure-criterion} now proves
closure for both products, while the norm and trace conditions are inherited
from the original octonionic orders.
\end{proof}

The metric information and a first multiplicative invariant are summarized in
Table~\ref{tab:four-integral-okubo-systems}. Here $L(1)$ denotes the set of
elements of norm one, and the last column counts the elements satisfying
$x*x=x$. The counts follow by solving the positive definite equation $n(x)=1$
in each displayed basis and testing the resulting finite set under the exact
product (\ref{eq:petersson-okubo-product}).
\begin{table}
\centering
\begin{tabular}{|c|c|c|c|c|}
\hline
Order & Metric lattice & Gram determinant & $|L(1)|$ & Okubo idempotents\\
\hline\hline
$L_{C}$ & $C_{8}$ & $2^{8}$ & $16$ & $4$\\
\hline
$L_{A}$ & $A_{2}^{4}$ & $3^{4}$ & $24$ & $3$\\
\hline
$L_{D}$ & $D_{4}^{2}$ & $4^{2}$ & $48$ & $9$\\
\hline
$L_{E}$ & $E_{8}$ & $1$ & $240$ & $12$\\
\hline
\end{tabular}
\caption{\label{tab:four-integral-okubo-systems}The four lattices preserved by
the integral Okubo product. The Gram form is $\langle x,x\rangle=2n(x)$.}
\end{table}

\section{\noun{The octonionic, para-octonionic and Okubo $E_{8}$ structures}}
\label{sec:three-E8-structures}

Proposition~\ref{prop:four-integral-systems} reveals a remarkable phenomenon on
$E_{8}$: one additive lattice $L_{E}$ carries three different products.

\begin{prop}[One lattice, three arithmetic algebras]
\label{prop:one-lattice-three-algebras}The identity map on $L_{E}$ is an
isometry among the octonionic, para-octonionic, and Okubo integral systems. In
particular, all three systems have the same Gram matrix, the same $240$ roots,
and the same Gosset polytope $4_{21}$. They are pairwise non-isomorphic as
algebras.
\end{prop}

\begin{proof}
The para-Hurwitz and Petersson constructions change only the multiplication.
Their norm is the original octonionic norm, so the additive group and the Gram
form are exactly the same in all three cases. Hence the vectors of norm one are
the same $240$ vectors listed in (\ref{eq:InvertibleElements}). On the other
hand, the octonionic product is unital and alternative; the para-octonionic
product is non-unital and non-alternative, but has the paraunit $1$; and the
Okubo product is neither unital nor para-unital and is non-alternative. These
properties are invariant under algebra isomorphisms.
\end{proof}

The difference can already be seen on the common root set. With octonionic
multiplication the $240$ roots form the Moufang loop $M_{240}(E_{8})$ and $1$
is its identity. With the para-octonionic multiplication the same set is the
para-isotope of this Moufang loop: it has the paraunit $1$, satisfying
$1\bullet x=x\bullet1=\overline{x}$, but it has no identity. With the Okubo
multiplication the same finite set is again closed, but has neither an identity
nor a paraunit. In both non-unital cases the root set is a quasigroup. Indeed,
for a fixed root $u$, left and right multiplication by $u$ are injective because
the corresponding real composition algebra is a division algebra; on a finite
set they are therefore bijective.

The idempotents give a simple numerical distinction. The octonionic division
algebra has only the nonzero idempotent $1$. Among the common $240$ roots, $57$
are para-octonionic idempotents: besides $1$, they are the $56$ integral roots
of the form
\begin{equation}
x=-\frac{1}{2}+v,\qquad v\in\operatorname{Im}(\mathbb{O}),\qquad
n(v)=\frac{3}{4}.
\end{equation}
For the Okubo product there are $12$ integral idempotents. Thus the geometry is
unchanged, but the three multiplications organise its vertices in different
ways.

There is also a difference between continuous and arithmetic symmetries. The
real octonion and para-octonion algebras both have automorphism group
$\mathrm{G}_{2}$, whereas the real Okubo algebra has automorphism group
$\mathrm{SU}(3)$. These full compact groups do not preserve a chosen $E_{8}$
order in general. For the present embeddings the relevant finite groups are
\begin{align}
G_{E}^{\mathrm{oct}}=G_{E}^{p\mathrm{oct}}
&=\{g\in\mathrm{G}_{2}:g(L_{E})=L_{E}\},\\
G_{E}^{\mathrm{Ok}}&=\{g\in\mathrm{SU}(3):g(L_{E})=L_{E}\}.
\end{align}
The first equality follows because octonion automorphisms commute with
conjugation and therefore preserve the para-product. Thus the octonionic and
para-octonionic arithmetic stabilizers coincide.

Every octonion automorphism preserving $L_E$ and commuting with $\tau$
preserves the Petersson product. Consequently,
\begin{equation}
C_{G_{E}^{\mathrm{oct}}}(\tau)\subseteq G_{E}^{\mathrm{Ok}},
\end{equation}
and $\tau$ itself gives a nontrivial element of the Okubo arithmetic stabilizer.
Determining whether this inclusion accounts for the complete stabilizer is a
separate finite group problem.

\begin{rem}
The rotation $\rho$ and the integral automorphism $\tau$ both give real Okubo
algebras, but they define different arithmetic embeddings with respect to the
standard octonionic orders. For $\rho$ one sees $\sqrt{3}$ and loses the four
lattices; for $\tau$ the matrices are integral and all four lattices are
preserved. The algebraic isomorphism class alone therefore does not determine
the integral order.
\end{rem}

\section{\noun{Conclusions and Future Developments}}\label{sec:conclusions}

In this work we have shown that the obstruction found in the old Okubo
presentation was not intrinsic. The rotation automorphism $\rho$ is
algebraically correct and produces the compact real Okubo algebra, but it does
not preserve the classical integral octonion orders. The signed automorphism
$\tau$ is again of order three, has a quaternionic fixed subalgebra, and its
Petersson isotope is therefore an Okubo algebra. Unlike $\rho$, it is defined
over $\mathbb{Z}$ and preserves the four systems $C_{8}$, $A_{2}^{4}$,
$D_{4}^{2}$, and $E_{8}$.

The most important consequence is that the Coxeter--Dickson lattice is itself
an integral order for all three products. The octonionic, para-octonionic, and
Okubo $E_{8}$ structures have the same Euclidean lattice, the same $240$ roots,
and the same Gosset polytope. They differ only after multiplication is
introduced: the octonionic roots form a unital Moufang loop, the
para-octonionic roots form a quasigroup with a paraunit and $57$ idempotents,
and the Okubo roots form a quasigroup with neither unit nor paraunit and $12$
idempotents.

The main similarities and differences are summarized in
Table~\ref{tab:conclusion-three-algebras}.
\begin{table}[!ht]
\centering
\small
\begin{tabular}{|>{\raggedright\arraybackslash}p{2.7cm}|
>{\raggedright\arraybackslash}p{3.15cm}|
>{\raggedright\arraybackslash}p{3.15cm}|
>{\raggedright\arraybackslash}p{3.15cm}|}
\hline
Property & Octonions & Para-octonions & Okubo algebra\\
\hline\hline
Underlying lattice & $E_{8}$ & $E_{8}$ & $E_{8}$\\
\hline
Product & $x\cdot y$ & $\overline{x}\cdot\overline{y}$ &
$\tau(\overline{x})\cdot\tau^{2}(\overline{y})$\\
\hline
Unit structure & Unit $1$ & Paraunit $1$ & Neither unit nor paraunit\\
\hline
Algebraic laws & Alternative and flexible & Non-alternative and flexible &
Non-alternative and flexible\\
\hline
Real automorphism group & $\mathrm{G}_{2}$ & $\mathrm{G}_{2}$ &
$\mathrm{SU}(3)$\\
\hline
Norm-one roots & Moufang loop $M_{240}(E_{8})$ & Para-isotope of
$M_{240}(E_{8})$ & Okubo quasigroup\\
\hline
Integral idempotents in $E_{8}$ & $1$ & $57$ & $12$\\
\hline
\end{tabular}
\caption{\label{tab:conclusion-three-algebras}Comparison of the three integral
algebra structures carried by the same $E_{8}$ lattice.}
\end{table}

It would be interesting to determine the full group $G_{E}^{\mathrm{Ok}}$, its
orbits on the $240$ roots and on the $12$ idempotents, and its relation with the
centralizer of $\tau$ in the octonionic arithmetic group. A systematic study of
other order-three automorphisms preserving maximal octonion orders may also
clarify how many non-isomorphic integral Okubo products can live on the same
$E_{8}$ lattice.
\section{\noun{Acknowledgments}}

The author thanks Alessio Marrani and Francesco Zucconi for discussions and
for the stimulating questions on Okubo algebras. He also thanks Raymond
Aschheim for discussions on integral numbers, and Richard Clawson, David
Chester, and Klee Irwin for discussions on lattices and root systems.

\end{document}